\documentclass[a4paper, 11pt]{article}
\usepackage[utf8]{inputenc}
\usepackage{amsmath, amsfonts, amssymb, amsthm}
\usepackage{graphicx, subfig}
\usepackage{color} 

\usepackage{epstopdf} 

\newtheorem{lemma}{Lemma}
\newtheorem{proposition}{Proposition}

\newtheorem{algo}{Algorithm}

\DeclareMathOperator{\id}{Id}

\newcommand{\ta}[1]{T_{#1}\Theta}
\newcommand{\et}{\boldsymbol{\eta}}
\newcommand{\mh}{\mathcal{H}}
\newcommand{\ti}{T}
\newcommand{\pa}[2]{\frac{\partial }{\partial {#2}}\ti_{#1}(\eta)}
\newcommand{\pav}[2]{\frac{\partial }{\partial {#2}}\ti_{#1}(\eta^{#1})}

\newcommand{\ngm}[2]{montee_modele_#1_algo_#2.tex}
\newcommand{\ngl}[2]{logeta_modele_#1_algo_#2.tex}
\newcommand{\ngr}[2]{diffregret_modele_#1_algo_#2.tex}

\newcommand{\ngmm}[2]{montee_meilleur_modele_#1_algo_#2.tex}

\newcommand{\ngcvinvpvp}[2]{cv_inv_pvp_modele_#1_algo_#2.tex}

\newcommand{\nom}[2]{#1_#2}

\begin{document}

\title{Speed learning on the fly}
\author{Pierre-Yves Massé, Yann Ollivier}
\date{}

\maketitle

\newcommand{\deq}{\mathrel{\mathop:}=}
\newcommand{\eqd}{=\mathrel{\mathop:}}

\begin{abstract}
The practical performance of online stochastic gradient descent
algorithms is highly dependent on the chosen step size, which must be
tediously hand-tuned in many applications. The same is true for more
advanced variants of stochastic gradients, such as SAGA, SVRG, or
AdaGrad. Here we propose to adapt the step size by performing a
gradient descent on the step size itself, viewing the whole performance of
the learning trajectory as a function of step size. Importantly, this
adaptation can be computed online at little cost, without having to iterate backward
passes over the full data.
\end{abstract}

\section*{Introduction}


This work aims at improving gradient ascent procedures for use in machine learning contexts, by adapting the step size of the descent as it goes along.

Let $\ell_0, \ell_1, \, \ldots, \, \ell_t,\,\ldots$ be functions to be maximised over some
parameter space $\Theta$. At each time $t$, we wish to compute or approximate
the parameter $\theta_t^\ast\in\Theta$ that maximizes the sum
\begin{equation}
	\label{eq:cumloss}
	L_t(\theta)\deq \sum_{s \leq t} \ell_s(\theta).
\end{equation}
In the experiments below, as in many applications, $\ell_t(\theta)$ writes
$\ell(x_t,\theta)$ for some data $x_0, x_1, \, \ldots, \, x_t,\,\ldots$

A common strategy, especially with large data size or
dimensionality \cite{Bottou_sgd},
is the online stochastic gradient ascent (SG)
\begin{equation}
\label{eqrectheta} 
	\theta_{t+1}=\theta_t + \eta \, \partial_\theta \ell_t(\theta_t)
\end{equation}
with step size $\eta$,
where
$\partial_\theta \ell_t$ stands for the Euclidean gradient of $\ell_t$ with
respect to $\theta$.

Such an approach has become a mainstay of both the
optimisation and machine learning communities \cite{Bottou_sgd}. Various conditions for
convergence exist, starting with the celebrated article of Robbins
and Monro \cite{RM}, or later \cite{kc}.
Other types of results are proved in convex settings,

Several variants have since been introduced, in part to
improve the convergence of the algorithm, which is much slower in
stochastic than than in deterministic settings. 
For instance, algorithms such as SAGA, 
Stochastic Variance Reduced Gradient (SVRG) or Stochastic Average
Gradient (SAG) \cite{saga,svrg,bsag},
perform iterations using a comparison between the latest gradient and an average of past gradients. This
reduces the variance of the resulting estimates and allows for nice
convergence theorems \cite{saga,bsag}, provided a reasonable step size
$\eta$ is used.

\paragraph{Influence of the step size.}
The ascent requires a parameter, the step size $\eta$, usually called
``learning rate" in the machine learning community. Empirical
evidence highligting the sensitivity of the ascent to its actual
numerical value exists aplenty; see for instance the graphs in
Section~\ref{sec:trajtheta}. Slow and tedious hand-tuning is therefore mandatory
in most applications. Moreover, admittable values of $\eta$ depend on the
parameterisation retained---except for descents described in terms of
Riemannian metrics \cite{Amari98}, which provide some degree of
parameterisation-invariance.

Automated procedures for setting reasonable value of $\eta$ are therefore
of much value. For instance, AdaGrad \cite{adagrad} divides the derivative
$\partial_\theta \ell_t$ by a root mean square average of the magnitude
of its recent values, so that the steps are of size approximately $1$;
but this still requires a ``master step size'' $\eta$.

Shaul, Zhang and LeCun in \cite{pesky} study a simple separable quadratic loss
model and compute the value of $\eta$ which minimises the expected loss
after each parameter update. This value can be expressed in terms of
computable quantities depending on the trajectory of the descent. These
quantities still make sense for non-quadratic models, making this idea
amenable to practical use.


More recently, Maclaurin, Douglas and Duvenaud \cite{gbho} propose to
directly conduct a gradient ascent on the hyperparameters (such as the
learning rate $\eta$) of any algorithm. The gradients with respect to the
hyperparameters are computed
exactly by ``chaining derivatives backwards through the entire training
procedure'' \cite{gbho}. Consequently, this approach is extremely
impractical in an online setting, as it optimizes the learning rate by
performing several passes, each of which goes backwards from time $t$ to
time $0$.

\paragraph{Finding the best step size.} The ideal value of the step size $\eta$ would be the one that maximizes
the cumulated objective function \eqref{eq:cumloss}.
Write $\theta_t(\eta)$ for the parameter value obtained after $t$
iterations of the gradient step \eqref{eqrectheta} using a given value $\eta$, and consider the sum
\begin{equation}
\label{eqcumloss}
	\sum_{s \leq t} \ell_s(\theta_s(\eta)).
\end{equation}
Our goal is to find an online way to approximate the value of $\eta$ that
provides the best value of this sum. This can be viewed as an ascent on the space of stochastic ascent
algorithms. 

We suggest to update $\eta$ through a stochastic gradient ascent
on this sum:
\begin{equation}
	\label{equeta} 
	\eta \leftarrow \eta + \alpha \frac{\partial}{\partial \eta} \ell_t(\theta_t(\eta))
\end{equation}
and then to use, at each time, the resulting value of $\eta$ for the next
gradient step \eqref{eqrectheta}.



The ascent \eqref{equeta} on $\eta$ depends, in turn, on a step size
$\alpha$. Hopefully,
the dependance on $\alpha$ of the whole procedure is somewhat lower than
that of the original stochastic gradient scheme on its step size $\eta$.

This approach immediately extends to other stochastic gradient
algorithms; in what follows we apply it both to the standard SG ascent
and to the SVRG algorithm.

The main point in this approach is to find efficient ways to compute or
approximate the derivatives
$\frac{\partial}{\partial \eta} \ell_t(\theta_t(\eta))$. Indeed, the value
$\theta_t(\eta)$ after $t$ steps depends on the
whole trajectory of the algorithm, and so does its derivative with
respect to $\eta$.

After reviewing the setting for gradient ascents in Section~\ref{sec:sg}, 
in Section~\ref{sec:quadraticalg} we provide an exact but impractical way of computing the
derivatives $\frac{\partial}{\partial \eta} \ell_t(\theta_t(\eta))$. 
Sections~\ref{sec:SGSG}--\ref{sec:SGAG} contain
the main contribution: SG/SG and SG/AG, practical algorithms to adjust $\eta$
based on two approximations with respect to these exact
derivatives.

Section~\ref{sec:generalLLR} extends this to other base algorithms
such as SVRG.
In Section~\ref{sec:fllr} one of the approximations is
justified by showing that it computes a derivative, not with respect to a
fixed value of $\eta$ as in \eqref{equeta}, but with respect to the sequences of values of
$\eta$ effectively used along the way. This also suggests improved
algorithms.

Section~\ref{sec:exp} provides experimental comparisons of gradient ascents
using traditional algorithms with various values of $\eta$, and the same
algorithms where $\eta$ is self-adjusted according to our scheme. The
comparisons are done on three sets of synthetic data: a one-dimensional
Gaussian model, a one-dimensional Bernoulli model and a 50-dimensional
linear regression model: these simple models already exemplify
the strong dependence of the traditional algorithms on the value of
$\eta$.

\paragraph{Terminology.} We say that an algorithm is of type ``LLR'' for
``Learning the Learning Rate'' when it updates its step size
hyperparameter $\eta$ as it unfolds.  We refer to LLR algorithms
by a compound abbreviation: ``SVRG/SG'', for instance, for an
algorithm which updates its parameter $\theta$ through SVRG and its
hyperparameter $\eta$ through an SG algorithm on $\eta$. 

\section{The Stochastic Gradient algorithm}
\label{sec:sg}
To fix ideas, we define the Stochastic Gradient (SG) algorithm as
follows.
In all that follows, $\Theta=\mathbb{R}^n$ for some
$n$.\footnote{$\Theta$ may also be any Riemannian manifold, a natural
setting when dealing with gradients.
Most
of the text is written in this spirit.}
The functions $\ell_t$ are assumed to be smooth.
In all our algorithms, the index $t$ starts at $0$.
\begin{algo}[Stochastic Gradient]
We maintain $\theta_t \in \Theta$ (current parameter), initialised at some arbitrary $\theta_0 \in \Theta$.
We fix $\eta \in \mathbb{R}$.
At each time $t$, we fix a rate $f(t) \in \mathbb{R}$.
The update equation reads:
\begin{equation}
  \label{eqrecthetaf} 
 \theta_{t+1} = \theta_t + \frac{\eta}{f(t)} \partial_\theta \ell_t(\theta_t). \\
\end{equation}
\end{algo}
The chosen rate $f(t)$ usually satisfies the well-known Robbins--Monro conditions \cite{RM}:
\begin{equation}
  \sum_{t \geq 0} f(t)^{-1} = \infty, \hspace{.5 cm} \sum_{t \geq 0} f(t)^{-2} < \infty.
\end{equation}
The divergence of the sum of the rates allows the ascent to go anywhere
in parameter space, while the convergence of the sum of the squares
ensures that variance remains finite.
Though custom had it that small such rates should be chosen, such as $f(t)=1/t$, recently the trend bucked towards the use of large ones, to allow for quick exploration of the parameter space. 
Throughout the article and experiments we use one such rate:
\begin{equation}
f(t)=\sqrt{t+2} \log(t+3).
\end{equation}
\section{Learning the learning rate on a stochastic gradient algorithm}
\label{secllrsg} 
\subsection{The loss as a function of step size}
\label{sec:quadraticalg}
To formalise what we said in the introduction, let us define, for each $\eta \in \mathbb{R}$, the sequence
\begin{equation}
\left( \theta_0, \theta_1, \theta_2, \ldots \right)
\end{equation}
obtained by iterating \eqref{eqrecthetaf} from some initial value $\theta_0$. 
Since they depend on $\eta$, we introduce, for each $t > 0$, the operator
\begin{equation}
\ti_t : \eta \in \mathbb{R} \mapsto \ti_t(\eta) \in \Theta,
\end{equation}
which maps any $\eta \in \mathbb{R}$ to the parameter $\theta_t$ obtained after $t$ iterations of \eqref{eqrecthetaf}. $\ti_0$ maps every $\eta$ to $\theta_0$. For each $t \geq 0$, the map $\ti_t$ is a regular function of $\eta$.
As explained in the introduction, we want to optimise $\eta$ according to the function:
\begin{equation}
\mathcal{L}_t(\eta) \deq \sum_{s \leq t} \ell_s(\ti_s(\eta)),
\end{equation}
by conducting an online stochastic gradient ascent on it.
We therefore need to compute the derivative in \eqref{equeta}:
\begin{equation}
  \frac{\partial}{\partial \eta} \ell_t(\ti_t(\eta)).
\end{equation}
To act more decisively on the order of magnitude of $\eta$, we perform an
ascent on its logarithm, so that we actually need to compute\footnote{
This is an abuse of notation as $\ti_t$
is not a function of $\log \eta$ but of $\eta$. Formally, we would need
to replace $\ti_t$ with $\ti_t \circ \exp$, which we refrain from
doing to avoid burdensome notation.}:
\begin{equation}
  \frac{\partial}{\partial \log \eta} \ell_t(\ti_t(\eta)).
\end{equation}
Now, the derivative of the loss at time $t$ with respect to $\eta$ can be
computed as the product of the derivative of $\ell_t$ with respect to
$\theta$ (the usual input of SG) and the derivative of $\theta_t$ with
respect to $\eta$:
\begin{equation}
\frac{\partial}{\partial \log \eta} \ell_t(\ti_t(\eta)) = \partial_\theta \ell_t(\ti_t(\eta)) \cdot A_t(\eta)
\end{equation}
where
\begin{equation}
	A_t(\eta) \deq \frac{\partial \ti_t(\eta)}{\partial \log \eta}. 
\end{equation}
Computation of the quantity $A_t$ and its approximation $h_t$ to be
introduced later, are the main focus of this text.
\begin{lemma}
\label{eqrecA} 
The derivative $A_t(\eta)$ may be computed through the following recursion equation.
$A_0(\eta)=0$ and, for $t\geq0$,
\begin{equation}
A_{t+1}(\eta)= A_t(\eta) + \frac{\eta}{f(t)} \partial_\theta
\ell_t(\ti_t(\eta)) + \frac{\eta}{f(t)} \, \partial_{\theta}^2 \ell_t(\ti_t(\eta)) \cdot A_t(\eta). 
\end{equation}
\end{lemma}
The proof lies in Section~\ref{cllrsg}.
This update of $A$ involves the Hessian of the loss function with respect
to $\theta$, evaluated in the direction of $A_t$. Often this quantity is
unavailable or too costly. Therefore we will
use a finite difference approximation instead: 
\begin{equation}
  \partial_{\theta}^2 \ell_t(\ti_t(\eta)) \cdot A_t(\eta) \approx \partial_\theta \ell_t \left(\ti_t(\eta) + A_t(\eta)\right) - \partial_\theta \ell_t(\ti_t(\eta)).
 \end{equation}
This design ensures that the resulting update on $A_t(\eta)$ uses the gradient
of $\ell_t$ only once:
\begin{equation}
\label{eq:Anohessian}
A_{t+1}(\eta) \approx A_t(\eta) + \frac{\eta}{f(t)} \partial_\theta \ell_t \left(\ti_t(\eta) + A_t(\eta) \right). 
\end{equation}

An alternative approach would be to compute the Hessian in the direction
$A_t$ by numerical differentiation.

\subsection{LLR on SG: preliminary version with simplified expressions (SG/SG)}
\label{sec:SGSG}
Even with the approximation above, computing the quantities $A_t$ would
have a quadratic cost in $t$: each time we update $\eta$ thanks to
\eqref{equeta}, 
we would need to compute anew all the $A_s(\eta), \, s \leq t$, as well
as the whole trajectory $\theta_t=T_t(\eta)$, at each
iteration $t$.
We therefore replace the $A_t(\eta)$'s by online approximations, the
quantities $h_t$, which implement the same evolution equation
$\eqref{eq:Anohessian}$ as $A_t$, disregarding the fact that $\eta$ may
have changed in the meantime. These quantities will be interpreted more properly in Section~\ref{sec:fllr} as
derivatives taken along the effective trajectory of the ascent.
This yields the SG/SG algorithm.
\begin{algo}[SG/SG]
	\label{asg2}
  We maintain $\theta_t \in \Theta$ (current parameter), $\eta_t \in
  \mathbb{R}$ (current step size) and $h_t \in \ta{\theta_t}$
  (approximation of the derivative of $\theta_t$ with respect to
  $\log(\eta)$).

  The first two are initialised arbitrarily,
  and $h_0$ is set to 0.

The update equations read:
\begin{equation}
	\left\lbrace \begin{aligned}
		\log \eta_{t+1} &= \log \eta_t + \frac{1}{\mu_t} \, \partial_\theta \ell_t(\theta_t) \cdot h_t \\
		h_{t+1} &= h_t + \frac{\eta_{t+1}}{f(t)} \, \partial_\theta \ell_t\left( \theta_t+ h_t \right) \\
		\theta_{t+1} &= \theta_t + \frac{\eta_{t+1}}{f(t)}
		\partial_\theta \ell_t(\theta_t)\,,
\end{aligned} \right.
\end{equation}
where $\mu_t$ is some learning rate on $\log \eta$, such as $\mu_t =
\sqrt{t+2}\log(t+3)$.
\end{algo}
\subsection{LLR on SG: efficient version (SG/AG)}
\label{sec:SGAG}
To obtain better performances, we actually use an adagrad-inspired scheme
to update the logarithm of the step size.
\begin{algo}[SG/AG]
	We maintain $\theta_t \in \Theta$ (current parameter), $\eta_t \in \mathbb{R}$ (current step size), $h_t \in \ta{\theta_t}$ (approximation of the derivative of $\theta_t$ with respect to $\log(\eta)$), $n_t \in \mathbb{R}$ (average of the squared norms of $\partial \ell_t \circ T_t / \partial \log \eta)$, and $d_t \in \mathbb{R}$ (renormalising factor for the computation of $n_t$).

	$\theta$ et $\eta$ are initially set to $\theta_0$ and $\eta_0$, the other variables are set to $0$.

	At each time $t$, we compute $\mu_t \in \mathbb{R}$ (a rate used in several updates), and $\lambda_t \in \mathbb{R}$ (the approximate derivative of $\ell_t \circ \theta_t$ with respect to $\log(\eta)$ at $\eta_t$).

The update equations read:
\begin{equation}
	\left\lbrace \begin{aligned}
		\mu_t &= \sqrt{t+2}\log(t+3) \\
		\lambda_t &= l(\theta_t) \cdot h_t \\
	  d_{t+1} &=\left(1-\frac{1}{\mu_t} \right) d_t + \frac{1}{\mu_t} \\
		n_{t+1}^2 &= \left( \left(1-\frac{1}{\mu_t} \right) n_t^2 + \frac{1}{\mu_t} \lambda_t^2 \right) d_{t+1}^{-1} \\
	  \log \eta_{t+1} &= \log \eta_t + \frac{1}{\mu_t} \, \frac{\lambda_t}{n_{t+1}} \\
	  h_{t+1} &= h_t + \frac{\eta_{t+1}}{f(t)} \, \partial_\theta \ell_t\left(\theta_t+ h_t \right) \\
 \theta_{t+1} &= \theta_t + \frac{\eta_{t+1}}{f(t)} \, \partial_\theta \ell_t(\theta_t).
	\end{aligned} \right.
      \end{equation}
\end{algo}

\subsection{LLR on other Stochastic Gradient algorithms}
\label{sec:generalLLR}

The LLR procedure may be applied to any stochastic gradient algorithm of the form
\begin{equation}
\theta_{t+1}=F(\theta_t,\eta_t)
\end{equation}
where $\theta_t$ may store all the information maintained by the
algorithm, not necessarily just a parameter value.
Appendix~\ref{llrgen} presents the algorithm in this case.
Appendix~\ref{llrsvrg} presents SVRG/AG, which is the particular case of
this procedure applied to SVRG with an AdaGrad scheme for the update of $\eta_t$.

\section{Experiments on SG and SVRG} 
\label{sec:exp}
We now present the experiments conducted to test our procedure. We first describe the experimental set up, then discuss the results. 
\subsection{Presentation of the experiments}
We conducted ascents on synthetic data generated by three different probabilistic models: a one-dimensional Gaussian model, a Bernoulli model and a $50$-dimensional linear regression model. Each model has two components: a generative distribution, and a set of distributions used to approximate the former.
\paragraph{One Dimensional Gaussian Model.}
The mean and value of the Gaussian generative distribution were set to $5$ and $2$ respectively. Let us note $p_\theta$ the density of a standard Gaussian random variable. 
The function to maximise we used is:
\begin{equation}
	\ell_t(\theta) = \log p_\theta(x_t) = -\frac{1}{2}(x_t - \theta)^2.
\end{equation}
\paragraph{Bernoulli model.}
The parameter in the standard parameterisation for the Bernoulli model
was set to $p=0.3$, but we worked with a logit parameterisation
$\theta=\log(p/(1-p))$ for both
the generative distribution and the discriminative function. The latter
is then:
\begin{equation}
  \ell_t(\theta) = \theta \cdot x_t - \log \left( 1 + e^\theta \right). 
\end{equation}
\paragraph{Fifty-dimensional Linear Regression model.}
In the last model, we compute a fixed random matrix $M$. We then draw
samples $Z$ from a standard $50$-dimensional Gaussian distribution. We
then use $M$ to make random linear combinations $X=MZ$ of the coordinates
of the $Z$ vectors. Then we observe $X$ and try to recover first
coordinate of the sample $Z$. The solution $\theta^*$ is the first row of
the inverse of $M$. Note $Y$ the first coordinate of $Z$ so that the regression pair is $(X,Y)$.
We want to maximise:
\begin{equation}
  \ell_t(\theta) = -\frac{1}{2} \left( y_t - \theta \cdot x_t \right)^2,
\end{equation}
\bigskip
For each model, we drew 2500 samples from the data (7500 for the $50$-dimensional model), then conducted ascents on those with on the one hand the SG and SVRG algorithms, and on the other hand their LLR counterparts, SG/SG and SVRG/SG, respectively.
\subsection{Description and analysis of the results}
For each model, we present four different types of results. We start with
the trajectories of the ascents for several initial values of $\eta$ (in
the $50$-dimensional case, we plot the first entry of $\theta^T \cdot M$). 
Then we present the cumulated regrets. Next we show the evolution of the
logarithm of $\eta_t$ along the ascents for the LLR algorithms. Finally,
we compare this to trajectories of the non-adaptive algorithms with good initial values of $\eta$.
Each time, we present three figures, one for each model.

\begin{figure}[!h]
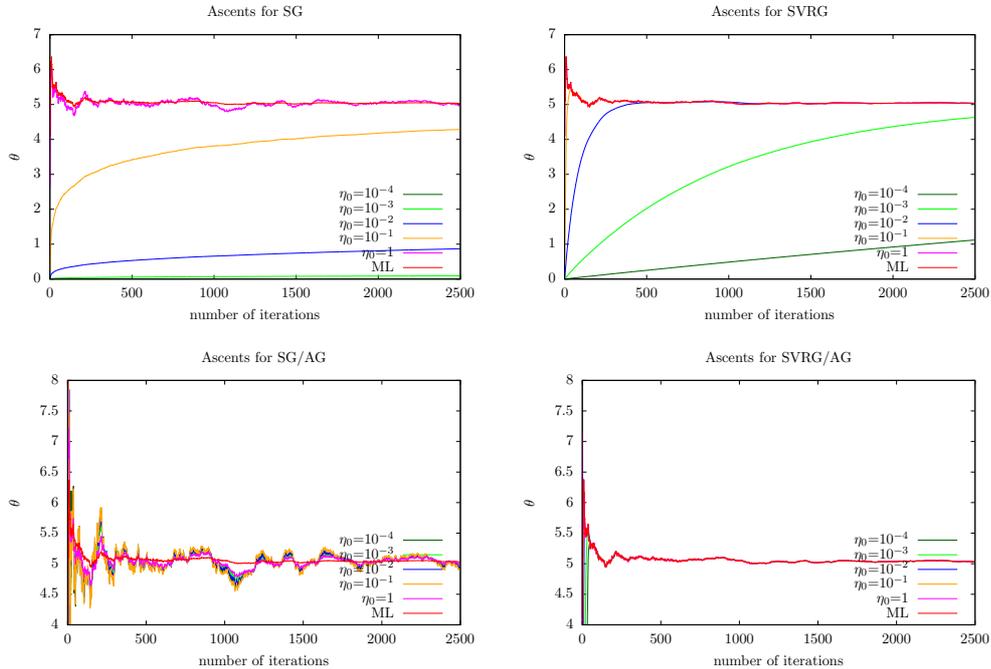

\centering
\begin{tabular}{cc}
  \scalebox{.5}{\input{\ngm{GAUSSIEN}{SG}}} &
  \scalebox{.5}{\input{\ngm{GAUSSIEN}{SVRG}}} \\ 
  \scalebox{.5}{\input{\ngm{GAUSSIEN}{\nom{SG}{SG}}}} &
  \scalebox{.5}{\input{\ngm{GAUSSIEN}{\nom{SVRG}{SG}}}}
\end{tabular}
\caption{Trajectories of the ascents for a Gaussian model in one dimension for several algorithms and several $\eta_0$'s}
\label{mg1d}
\end{figure}

Each figure of Figures~\ref{mg1d} to \ref{mrg50d} is made of four graphs: the upper ones are those of SG and SVRG, the lower ones are those of SG/SG and SVRG/SG. Figures~\ref{mg1d} to \ref{mrg50d} present the trajectories of the ascents for several orders of magnitude of $\eta_0$, while Figures~\ref{rg1d} to \ref{rrg50d} present the cumulated regrets for the same $\eta_0$'s. The trajectory of the running maximum likelihood (ML) is displayed in red in each plot.
\subsubsection{Trajectories of $\theta$}
\label{sec:trajtheta}
\begin{figure}[!h]
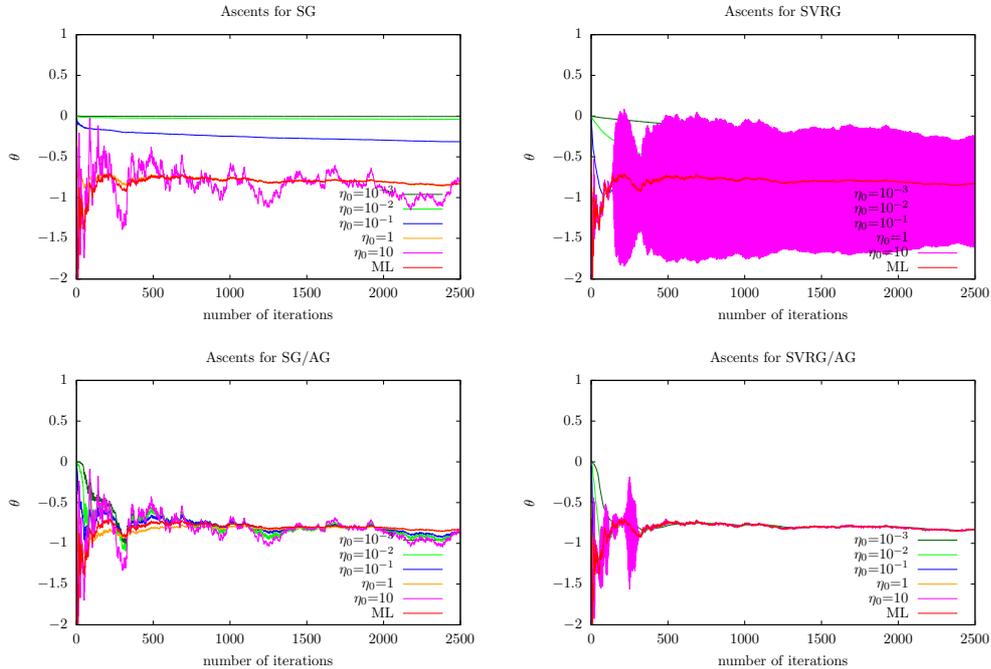

\centering
\begin{tabular}{cc}
  \scalebox{.5}{\input{\ngm{BERNOULLI}{SG}}} &
  \scalebox{.5}{\input{\ngm{BERNOULLI}{SVRG}}} \\ 
  \scalebox{.5}{\input{\ngm{BERNOULLI}{\nom{SG}{SG}}}} &
  \scalebox{.5}{\input{\ngm{BERNOULLI}{\nom{SVRG}{SG}}}}
\end{tabular}
\caption{Trajectories of the ascents for a Bernoulli model for several algorithms and several $\eta_0$'s}
\label{mb1d}
\end{figure}
Each figure for the ascent looks the same: there are several well distinguishable trajectories in the graphs of the standard algorithms, the upper ones, while trajectories are much closer to each other in those of the LLR algorithms, the lower ones.

Indeed, for many values of $\eta$, the standard algorithms will perform poorly. For instance, low values of $\eta$ will result in dramatically low convergence towards the ML, as may be seen in some trajectories of the SG graphs. The SVRG algorithm performs noticeably better, but may start to oscillate, as in Figures~\ref{mb1d} and \ref{mrg50d}.

These inconveniences are significantly improved by the use of LLR procedures. Indeed, in each model, almost every trajectory gets close to that of the ML in the SG/AG graphs. In the SVRG/AG graphs, the oscillations are overwhelmingly damped.
Improvements for SG, though significant, are not as decisive in the
linear regression model as in the other two, probably due to its greater complexity.
\begin{figure}[!h]
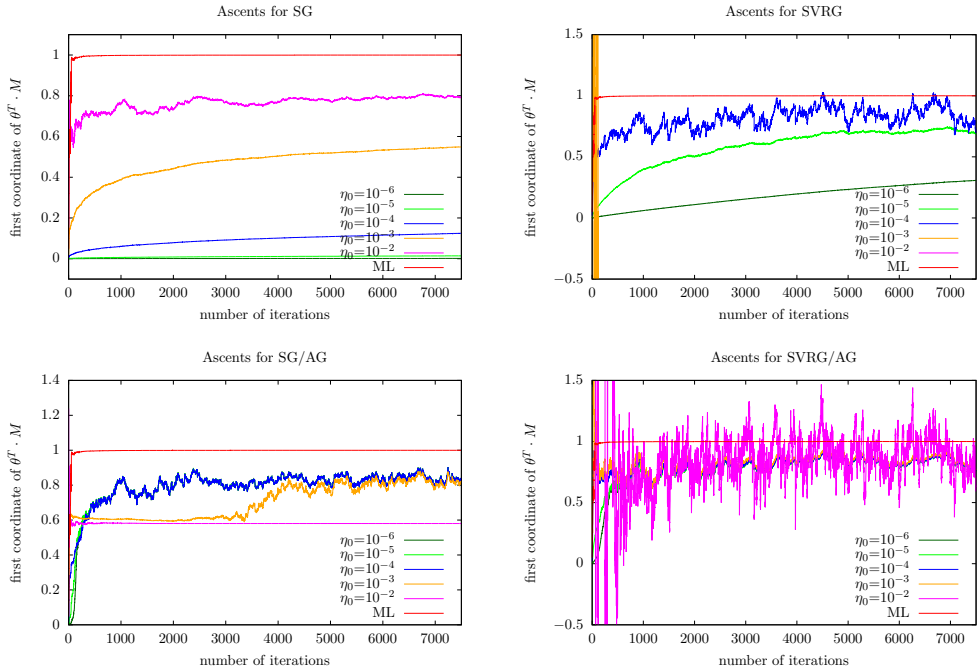

\centering
\begin{tabular}{cc}
  \scalebox{.5}{\input{\ngm{MOINDRES_CARRES}{SG}}} &
  \scalebox{.5}{\input{\ngm{MOINDRES_CARRES}{SVRG}}} \\ 
  \scalebox{.5}{\input{\ngm{MOINDRES_CARRES}{\nom{SG}{SG}}}} &
  \scalebox{.5}{\input{\ngm{MOINDRES_CARRES}{\nom{SVRG}{SG}}}}
\end{tabular}
\caption{Trajectories of the ascents for a $50$-dimensional linear regression model for several algorithms and several $\eta_0$'s}
\label{mrg50d}
\end{figure}
\begin{figure}[!h]
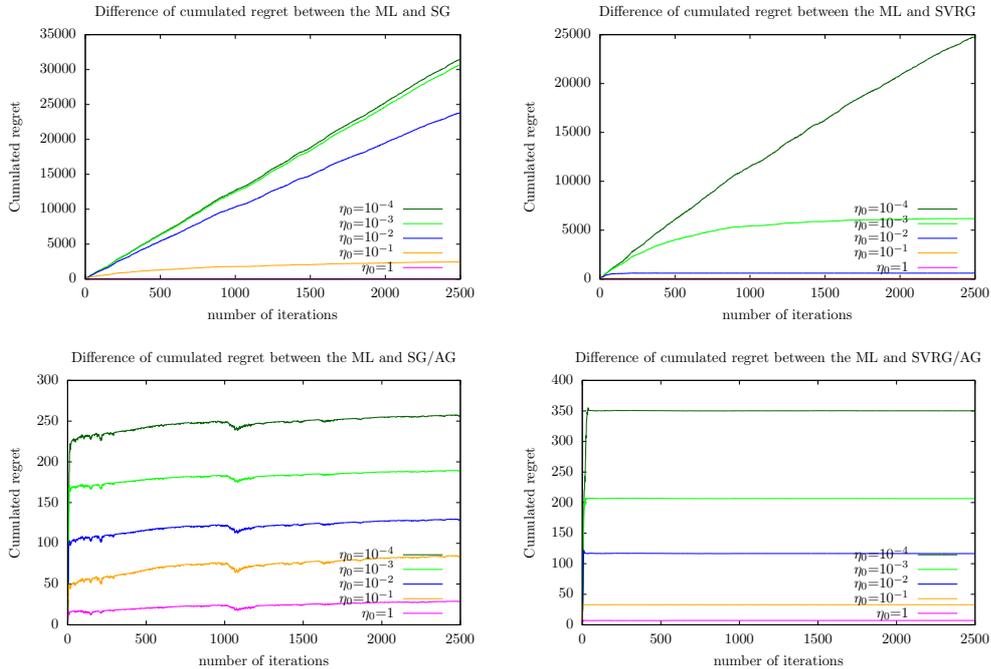

\centering
\begin{tabular}{cc}
  \scalebox{.5}{\input{\ngr{GAUSSIEN}{SG}}} &
  \scalebox{.5}{\input{\ngr{GAUSSIEN}{SVRG}}} \\ 
  \scalebox{.5}{\input{\ngr{GAUSSIEN}{\nom{SG}{SG}}}} &
  \scalebox{.5}{\input{\ngr{GAUSSIEN}{\nom{SVRG}{SG}}}}
\end{tabular}
\caption{Difference between the cumulated regrets of the algorithm and of the ML for a Gaussian model in one dimension for several algorithms and several $\eta_0$'s} 
\label{rg1d} 
\end{figure}
\subsubsection{Cumulated regrets}
\begin{figure}[!h]
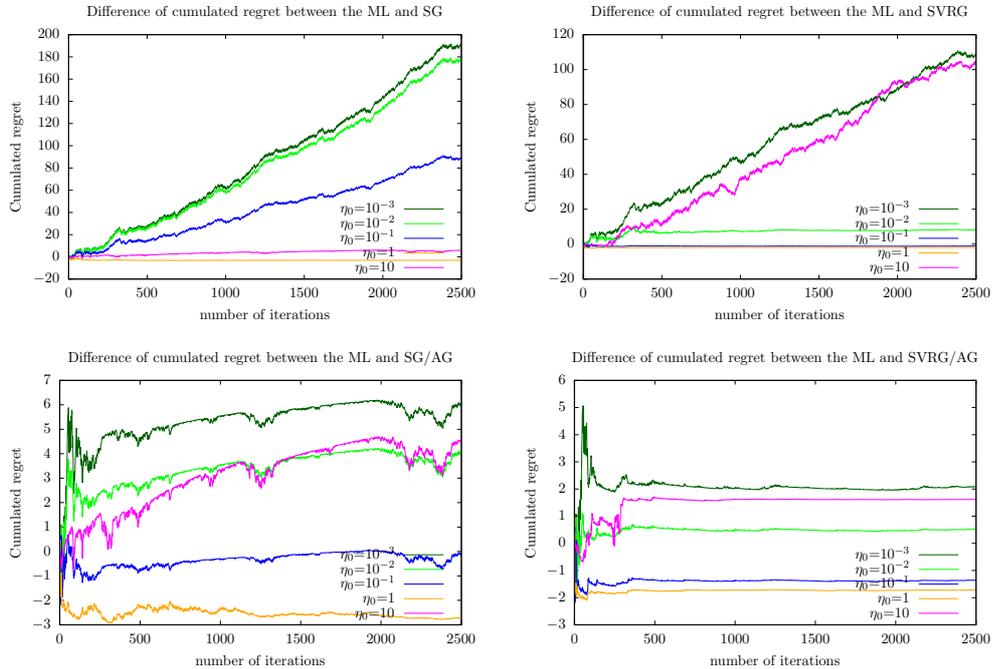

\centering
\begin{tabular}{cc}
  \scalebox{.5}{\input{\ngr{BERNOULLI}{SG}}} &
  \scalebox{.5}{\input{\ngr{BERNOULLI}{SVRG}}} \\ 
  \scalebox{.5}{\input{\ngr{BERNOULLI}{\nom{SG}{SG}}}} &
  \scalebox{.5}{\input{\ngr{BERNOULLI}{\nom{SVRG}{SG}}}}
\end{tabular}
\caption{Difference between the cumulated regrets of the algorithm and of the ML for a Bernoulli model for several algorithms and several $\eta_0$'s} 
\label{rb1d}
\end{figure}
Each curve of Figures~\ref{rg1d} to \ref{rrg50d} represents the difference between the cumulated regret of the algorithm used and that of the ML, for the $\eta_0$ chosen. The curves of SG and SVRG all go upwards, which means that the difference increases with time, whereas those of SG/AG and SVRG/SG tend to stagnate strikingly quickly. Actually, the trajectories for the linear regression model do not stagnate, but they are still significantly better for the LLR algorithms than for the original ones. The stagnation means that the values of the parameter found by these algorithms are very quickly as good as the Maximum Likelihood for the prediction task. Arguably, the fluctuations of the ascents around the later are therefore not a defect of the model: the cumulated regret graphs show that they are irrelevant for the minimisation at hand. 
\begin{figure}[!h]
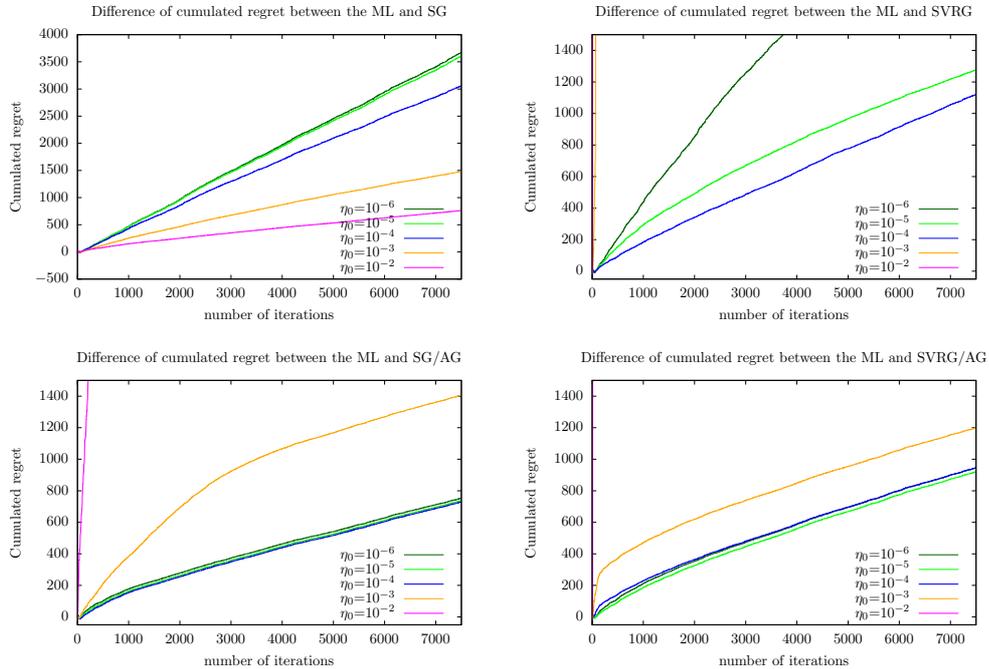

\centering
\begin{tabular}{cc}
  \scalebox{.5}{\input{\ngr{MOINDRES_CARRES}{SG}}} &
  \scalebox{.5}{\input{\ngr{MOINDRES_CARRES}{SVRG}}} \\ 
  \scalebox{.5}{\input{\ngr{MOINDRES_CARRES}{\nom{SG}{SG}}}} &
  \scalebox{.5}{\input{\ngr{MOINDRES_CARRES}{\nom{SVRG}{SG}}}}
\end{tabular}
\caption{Difference between the cumulated regrets of the algorithm and of the ML for a $50$-dimensional linear regression model for several algorithms and several $\eta_0$'s} 
\label{rrg50d}
\end{figure}
\begin{figure}[!h]
\label{lg1d}
\centering
\begin{tabular}{cc}
  \scalebox{.5}{\input{\ngl{GAUSSIEN}{\nom{SG}{SG}}}} &
  \scalebox{.5}{\input{\ngl{GAUSSIEN}{\nom{SVRG}{SG}}}} \\ 
  \scalebox{.5}{\input{\ngm{GAUSSIEN}{\nom{SG}{SG}}}} &
  \scalebox{.5}{\input{\ngm{GAUSSIEN}{\nom{SVRG}{SG}}}}
\end{tabular}
\caption{Evolution of $\log(\eta_t)$ in regard of the corresponding ascents for a Gaussian model in one dimension for $SG$ and $SVRG$ with LLR and several $\eta_0$'s}
\label{lg1d}
\end{figure}
\subsubsection{Evolution of the step size of the LLR algorithms during the ascents}
\begin{figure}[!h]
\centering
\begin{tabular}{cc}
  \scalebox{.5}{\input{\ngl{BERNOULLI}{\nom{SG}{SG}}}} &
  \scalebox{.5}{\input{\ngl{BERNOULLI}{\nom{SVRG}{SG}}}} \\ 
  \scalebox{.5}{\input{\ngm{BERNOULLI}{\nom{SG}{SG}}}} &
  \scalebox{.5}{\input{\ngm{BERNOULLI}{\nom{SVRG}{SG}}}}
\end{tabular}
\caption{Evolution of $\log(\eta_t)$ in regard of the corresponding ascents for a Bernoulli model in one dimension for SG and SVRG with LLR and several $\eta_0$'s}
\label{lb1d}
\end{figure}
Figures~\ref{lg1d} to \ref{lrg50d} show the evolution of the value of the logarithm of $\eta_t$ in the LLR procedures for the three models, in regard of the trajectories of the corresponding ascents. For the Gaussian and Bernoulli models, in Figures~\ref{lg1d} and \ref{lb1d}, $\log(\eta_t)$ tends to stagnate quite quickly. This may seem a desirable behaviour : the algorithms have reached good values for $\eta_t$, and the ascent may accordingly proceed with those. However, this analysis may seem somewhat unsatisfactory due to the $1/f(t)$ dampening term in the parameter update, which remains unaltered by our procedure. 
For the linear regression model, in Figure~\ref{lrg50d}, the convergence takes longer in the SG/SG case, and even in the SVRG/SG one, which may be explained again by the complexity of the model.
\begin{figure}[!h]
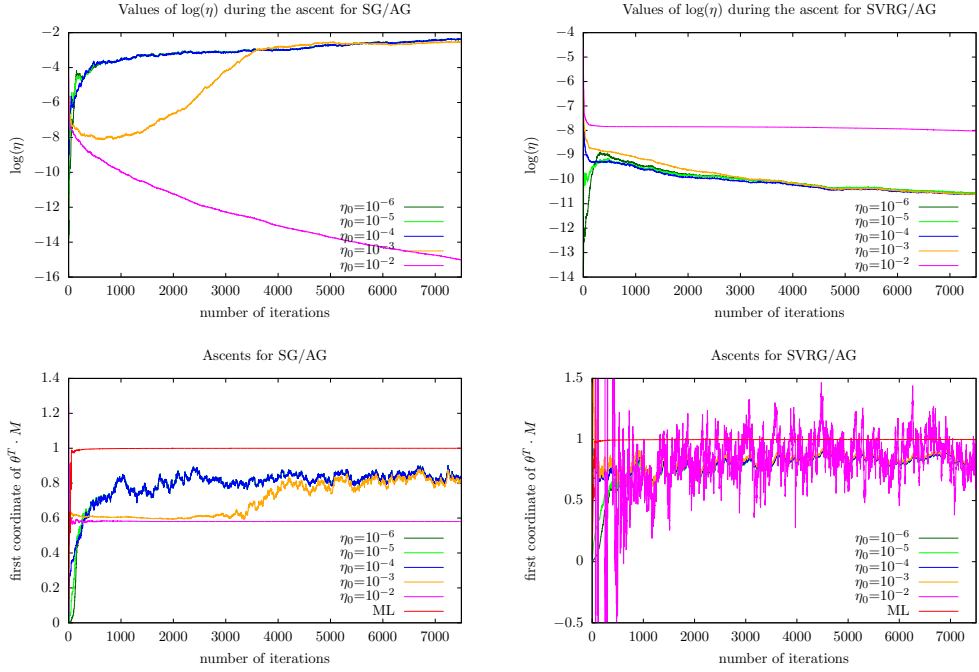

\label{lg50d}
\centering
\begin{tabular}{cc}
  \scalebox{.5}{\input{\ngl{MOINDRES_CARRES}{\nom{SG}{SG}}}} &
  \scalebox{.5}{\input{\ngl{MOINDRES_CARRES}{\nom{SVRG}{SG}}}} \\ 
  \scalebox{.5}{\input{\ngm{MOINDRES_CARRES}{\nom{SG}{SG}}}} &
  \scalebox{.5}{\input{\ngm{MOINDRES_CARRES}{\nom{SVRG}{SG}}}}
\end{tabular}
\caption{Evolution of $\log(\eta_t)$ in regard of the corresponding ascents for a $50$-dimensional linear regression model for SG and SVRG with LLR and several $\eta_0$'s}
\label{lrg50d}
\end{figure}
\begin{figure}[!h]
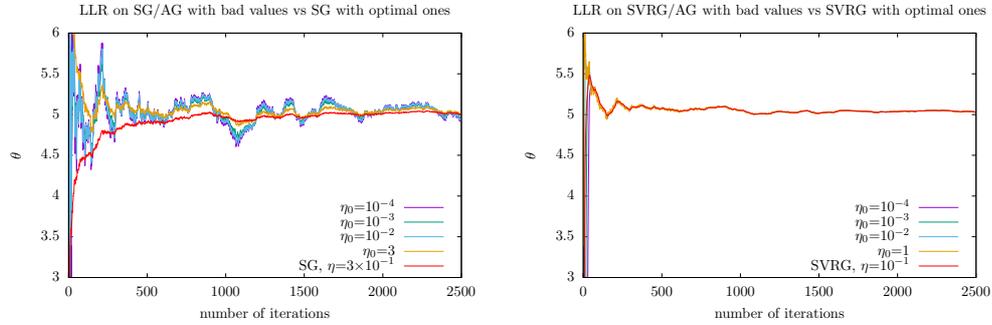

\centering
\begin{tabular}{cc}
  \scalebox{.5}{\input{\ngmm{GAUSSIEN}{\nom{SG}{SG}}}} &
  \scalebox{.5}{\input{\ngmm{GAUSSIEN}{\nom{SVRG}{SG}}}} \\ 
\end{tabular}
\caption{Trajectories of the ascents for a Gaussian model in one dimension for LLR algorithms with poor $\eta_0$'s and original algorithms with empirically optimal $\eta$'s}
\label{llrgvg} 
\end{figure}
\subsubsection{LLR versus hand-crafted learning rates}
\begin{figure}[!h]
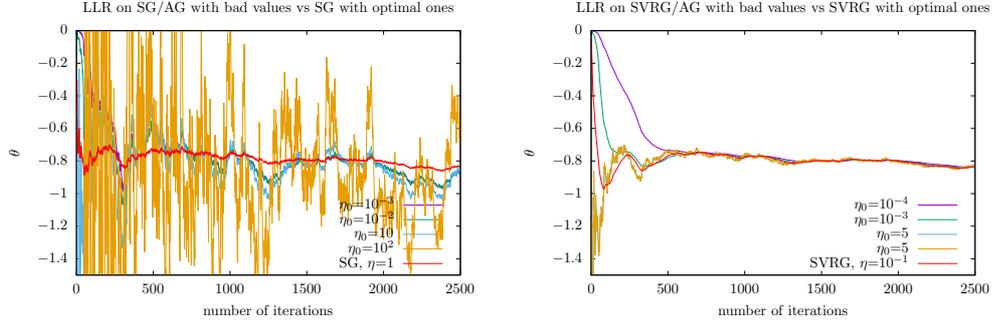

\centering
\begin{tabular}{cc}
  \scalebox{.5}{\input{\ngmm{BERNOULLI}{\nom{SG}{SG}}}} &
  \scalebox{.5}{\input{\ngmm{BERNOULLI}{\nom{SVRG}{SG}}}} \\ 
\end{tabular}
\caption{Trajectories of the ascents for a Bernoulli model for LLR algorithms with poor $\eta_0$'s and original algorithms with empirically optimal $\eta$'s}
\label{llrgvl} 
\end{figure}
Figures~\ref{llrgvg} to \ref{llrgvrg} show the trajectories of the ascents for LLR algorithms with poor initial values of the step size, compared to the trajectories of the original algorithms with hand-crafted optimal values of $\eta$. The trajectories of the original algorithms appear in red. They possess only two graphs each, where all the trajectories are pretty much undistinguishable from another. 
This shows that the LLR algorithms show acceptable behaviour even with poor initial values of $\eta$, proving the procedure is able to rescue very badly initialised algorithms.
However, one caveat is that the LLR procedure encounters difficulties dealing with too large values of $\eta_0$, and is much more efficient at dealing with small values of $\eta_0$. We have no satisfying explanation of this phenomenon yet.
We thus suggest, in practice, to underestimate rather than overestimate the initial
value $\eta_0$.
\begin{figure}[!h]
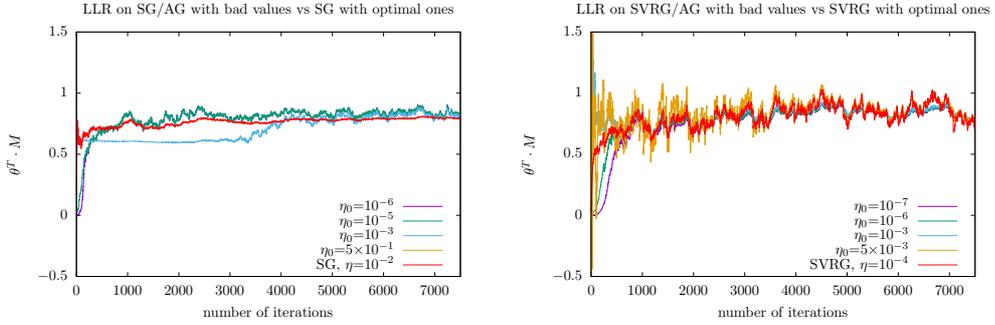

\centering
\begin{tabular}{cc}
  \scalebox{.5}{\input{\ngmm{MOINDRES_CARRES}{\nom{SG}{SG}}}} &
  \scalebox{.5}{\input{\ngmm{MOINDRES_CARRES}{\nom{SVRG}{SG}}}} \\ 
\end{tabular}
\caption{Trajectories of the ascents for a $50$-dimensional linear regression model for LLR algorithms with poor $\eta_0$'s and original algorithms with empirically optimal $\eta$'s}
\label{llrgvrg} 
\end{figure}
\subsection{$\eta_t$ in a quadratic model}
In a quadratic deterministic one-dimensional model, where we want to maximise:
\begin{equation}
  f(\theta)=-\alpha \frac{x^2}{2},
 \end{equation}
 SG is numerically stable if, and only if,
 \begin{equation}
	 \left\vert 1-\frac{\alpha \eta}{f(t)} \right\vert < 1,
\end{equation}
that is
 \begin{equation}
	 \frac{\eta}{2f(t)} < \alpha^{-1}.
\end{equation}
Each graph of Figure~\ref{cv_inv_pvp} has two curves, one for the original algorithm, the other for its LLR version. The curve of the LLR version goes down quickly, then much more slowly, while the other curve goes down slowly all the time. This shows that, for $\alpha=10^8$, the ratio above converges quickly towards $\alpha^{-1}$ for SG/AG and SVRG/AG, showing the ascent on $\eta$ is indeed efficient. Then, the algorithm has converged, and $\eta_t$ stays nearly constant, so much so that the LLR curve behaves like the other one. However, the convergence of $\eta_t$ happens too slowly: $\theta_t$ takes very large values before $\eta_t$ reaches this value, and even though it eventually converges to $0$, such behaviour is unacceptable in practise. 
\begin{figure}[!h]
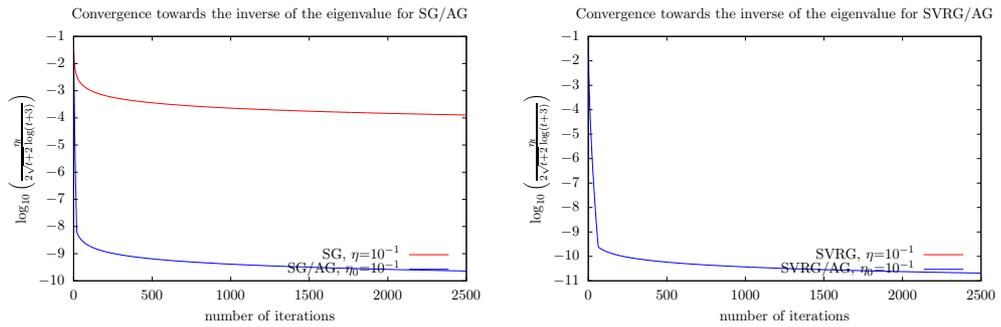

\begin{tabular}{cc}
  \scalebox{.5}{\input{\ngcvinvpvp{D_DIM}{\nom{SG}{SG}}}} &
  \scalebox{.5}{\input{\ngcvinvpvp{D_DIM}{\nom{SVRG}{SG}}}}
\end{tabular}
\caption{Evolution of $\log_{10} \left( \frac{\eta_t}{2\sqrt{t+2}\log(t+3)} \right)$ for a quadratic deterministic one-dimensional model for SG, SVRG and their LLR versions}
\label{cv_inv_pvp}
\end{figure}
\section{A pathwise interpretation of the derivatives}
\label{sec:fllr}
Until now, we have tried to optimise the step size for a stochastic
gradient ascent. This may be interpreted as conducting a gradient ascent
on the subspace of the ascent algorithms which gathers the stochastic
gradient algorithms, parametrised by $\eta \in \mathbb{R}$. However, we
had to replace the $A_t(\eta)$'s by the $h_t$'s because computing the
former gave our algorithm a quadratic complexity in time. Indeed,
adhesion to Equation~\ref{eqrecA} entails using $A_0(\eta_1)$ to compute
$A_1(\eta_1)$, for instance. Likewise, $A_0(\eta_2)$ and $A_1(\eta_2)$
would be necessary to compute $A_2(\eta_2)$, and this scheme would repeat
itself for every iteration.

We now introduce a formalism which shows the approximations we use are
actually derivatives taken alongside the effective trajectory of the
ascent.  It will also allow us to devise a new algorithm. It will, however, not account for the approximation of the
Hessian.

To this avail, let us parameterise stochastic gradient algorithms by a sequence of step sizes 
\begin{equation}
	\et = \left( \eta_0, \eta_1, \ldots \right)
\end{equation} 
such that at iteration $t$, the update equation for $\theta_t$ becomes:
\begin{equation}
  \label{eqrectheta2f} 
  \theta_{t+1} = \theta_t + \frac{\eta_{t+1}}{f(t)} \partial_\theta \ell_t(\theta_t). \\
\end{equation}
\subsection{The loss as a function of step size: extension of the formalism}
Consider the space $\mathcal{S}$ of infinite real sequences
\begin{equation}
	\et=\left( \eta_0, \eta_1, \eta_2, \ldots \right)
 \end{equation}
We expand the $\ti_t$ operators defined in Section~\ref{secllrsg} to
similar ones defined on $\mathcal{S}$, with the same notation. Namely, define $\ti_0(\et)=\theta_0$ and, for $t>0$,
\begin{equation}
	T_t : \et \in \mathcal{S} \mapsto \ti_t(\et) \in \mathbb{R}
\end{equation}
where $\theta_t$ has been obtained thanks to t iterations of \eqref{eqrectheta2f}.
$T_t$ is a regular function of $\et$, as the computations only involve
 \begin{equation}
 \eta_0, \eta_1, \ldots, \eta_t,
 \end{equation}
 and so take place in finite-dimensional spaces. This will apply in all the computations below. 
As before, we work on a space we call $\log(\mathcal{S})$, the image of $\mathcal{S}$ by the mapping 
\begin{equation}
\et=(\eta_t)_{t \geq 0} \in \mathcal{S} \mapsto \log(\et)=(\log \eta_t)_{t \geq 0},
\end{equation}
but we do not change notation for the functions $\et \mapsto \ti_t(\et)$, as in Section~\ref{secllrsg}.
\subsection{The update of the step size in the SG/SG algorithm as a gradient ascent} 
\label{secjsg2} 
We now prove that in SG/SG, when the Hessian is used without approximations, the step size $\eta_t$ indeed follows a gradient ascent scheme.
\begin{proposition}
\label{pjsgsg} 
Let
\begin{equation}
	\left( \theta_t \right)_{t \geq 0}, \hspace{.5 cm} \et= \left( \eta_t \right)_{t \geq 0}
\end{equation}
be the sequences of parameters and step-sizes obtained with the SG/SG algorithm, where the Hessian is not approximated: this is Algorithm~\ref{asg2} where the update on $h_t$ is replaced with
\begin{equation}
	h_{t+1}= h_t + \frac{\eta}{\mu_t} \partial_\theta \ell_t(\theta_t) + \frac{\eta}{\mu_t} \, \partial_{\theta}^2 \ell_t(\theta_t) \cdot h_t. 
\end{equation}
Define $e$ in the tangent plane of $\log(\mathcal{S})$ at $\log(\et)$ by
\begin{equation}
	e_t=1, \hspace{.5 cm} t \geq 0.
\end{equation}
Then, for all $t \geq 0$,
\begin{equation}
	\log \eta_{t+1}=\log \eta_t + \frac{1}{\mu_t} \frac{\partial}{\partial e} \ell_t(T_t(\et)).
\end{equation}
\end{proposition}
The proof lies in Appendix~\ref{cjsg2}.
\subsection{A new algorithm, using a notion of ``memory'' borrowed from \cite{pesky}}
\label{secfum} 
We would now like to compute the change in $\eta$ implied by a small
modification of all previous coordinates $\eta_s$ for $s$ less than the
current time $t$, but to compute the modification differently according
to whether the coordinate $s$ is ``outdated'' or not. To do it, we use the quantity $\tau_t$ defined in Section~4.2 of \cite{pesky} as the ``number of samples in recent memory''. We want to discard the old $\eta$'s and keep the recent ones. Therefore, at each time $t$, we compute 
\begin{equation}
  \gamma_t = \exp(-1/\tau_t).
\end{equation}
Choose $\et \in \log(\mathcal{S})$,
and consider the vector in the tangent plane to $\log(\mathcal{S})$ at $\et$:
\begin{equation}
  e_t^j = \left\lbrace \begin{aligned}
    \prod_{k=j}^t \gamma_j, \hspace{.5 cm} j \leq t \\
    0, \hspace{.5 cm} j \geq t+1.
  \end{aligned} \right.
\end{equation}
To run an algorithm using the $e_t$'s instead of $e$ as before, all we need to compute again is the formula for the update of the derivative below:
\begin{equation}
  \mathcal{H}_t \deq \frac{\partial}{\partial e_t}\ti_t(\et).
 \end{equation}
 $\mathcal{H}_t$ may indeed be computed, thanks to the following result.
\begin{proposition}
\label{puhvp} 
The update equation of $\mathcal{H}_t$ is:
\begin{equation}
  \label{equm} 
  \mathcal{H}_{t+1}=\gamma_{t+1} \mathcal{H}_t+ \gamma_{t+1} \, \frac{\eta_{t+1}}{f(t)} \partial_\theta \ell_t(\ti_t(\eta^t)) + \gamma_{t+1} \frac{\eta_{t+1}}{f(t)} \partial_{\ti}^2 \ell_t(\ti_t(\et)) \cdot \mathcal{H}_t .
\end{equation}
\end{proposition}
The proof lies in Section~\ref{cfum}.
\section*{Acknowledgements}
The first author would like to thank Gaetan Marceau-Caron for his
advice on programming, and Jérémy Bensadon for crucial help with \LaTeX.
\bibliographystyle{alpha}
\bibliography{biblio}
\newpage
\appendix
\section{LLR applied to the Stochastic Variance Reduced Gradient}
\label{llrsvrg}
The Stochastic Variance Reduced Gradient (SVRG) was introduced by Johnson and Zhang in \cite{svrg}. We define here a version intended for online use.
\begin{algo}[SVRG online]
	We maintain $\theta_t, \theta^b \in \Theta$ (current parameter and base parameter) and $s^b_t \in \ta{\simeq \theta_t}$ (sum of the gradients of the $\ell_s$ computed at $\theta_s$ up to time $t$).

	$\theta$ is set to $\theta_0$ and $\theta^b$ along $s^b$ to $0$.

The update equations read: 
\begin{equation}
	\left\lbrace \begin{aligned}
		s^b_{t+1} &=s^b_t + \partial_\theta \ell_t(\theta^b) \\
		\theta_{t+1} &= \theta_t + \eta \left( \partial_\theta \ell_t(\theta_t) - \partial_\theta \ell_t(\theta^b) + \frac{s^b_{t+1}}{t+1} \right).
		\end{aligned} \right.
\end{equation}
\end{algo}
We now present the LLR version, obtained by updating the $\eta$ of SVRG thanks to an SG ascent. 
We call this algorithm ``SVRG/SG''.
\begin{algo}[SVRG/AG]
  We maintain $\theta_t, \theta^b \in \Theta$ (current parameter and base parameter), $\eta_t \in \mathbb{R}$ (current step size), $s^b_t \in \ta{\simeq \theta_t}$ (sum of the gradients of the $\ell_s$ computed at $\theta_s$ up to time $t$), $h_t \in \ta{\theta_t}$ (approximation of the derivative of $\ti_t$ with respect to $\log(\eta)$ at $\eta_t$) and the real numbers $n_t$ (average of the squared norms of the $\lambda_s$ defined below) and $d_t$ (renormalising factor for the computation of $n_t$).

$\theta$ is set to $\theta_0$, the other variables are set to $0$.

At each time $t$, we compute $\mu_t \in \mathbb{R}$ (a rate used in several updates), and $\lambda_t \in \mathbb{R}$ (the approximate derivative of $\ell_t \circ \theta_t$ with respect to $\log(\eta)$ at $\eta_t$).

The update equations read:
\begin{equation}
	\left\lbrace \begin{aligned}
		\mu_t &= \sqrt{t+2}\log(t+3) \\
		\lambda_t &= \partial_\theta \ell_t(\theta_t) \cdot h_t \\
	  d_{t+1} &=\left(1-\frac{1}{\mu_t} \right) d_t + \frac{1}{\mu_t} \\
		n_{t+1}^2 &= \left( \left(1-\frac{1}{\mu_t} \right) n_t^2 + \frac{1}{\mu_t} \lambda_t^2 \right) d_{t+1}^{-1} \\
		\eta_{t+1} &= \eta_t \exp \left( \frac{1}{\mu_t} \, \frac{\lambda_t}{n_{t+1}} \right) \\
		s^b_{t+1} &=s^b_t + \partial_\theta \ell_t(\theta^b) \\
		h_{t+1} &= h_t + \eta_{t+1} \left( \partial_\theta \ell_t(\theta_t+h_t) - \partial_\theta \ell_t(\theta^b)+\frac{s^b_{t+1}}{t+1} \right) \\
		\theta_{t+1} &= \theta_t + \eta_{t+1} \left( \partial_\theta \ell_t(\theta_t) - \partial_\theta \ell_t(\theta^b) + \frac{s^b_{t+1}}{t+1} \right). \\
	\end{aligned} \right.
\end{equation}
\end{algo}
\newpage
\section{LLR applied to a general stochastic gradient algorithm}
\label{llrgen} 
Let $\Theta$ and $H$ be two spaces. $\Theta$ is the space of parameters, $H$ is that of hyperparameters. In this section, a parameter potentially means a tuple of parameters in the sense of other sections. For instance, in SVRG/SG online, we would call a parameter the couple
\begin{equation}
\left( \theta_t, \theta^b \right).
\end{equation}
Likewise, in the same algorithm, we would call a hyperparameter the couple
\begin{equation}
	\left( \eta_t, h_t \right).
\end{equation}
Let
\begin{equation} 
\begin{aligned}
F &: &\Theta \times H &\to \Theta \\
& & \left(\theta, \eta \right) &\mapsto  F(\theta,\eta).
\end{aligned}
\end{equation}
be differentiable with respect to both variables. We consider the algorithm:
\begin{equation}
\theta_{t+1}=F(\theta_t,\eta_t).
\end{equation}
Let us present its LLR version. We call it GEN/SG, GEN standing for ``general''.
\begin{algo}[GEN/SG]
  We maintain $\theta_t \in \Theta$ (current parameter), $\eta_t \in H$ (current hyperparameter), $h_t \in \ta{\theta_t}$ (approximation of the derivative of $\ti_t$ in the direction of $e \in T_{\eta_t}H$).

$\theta$ and $\eta$ are set to user-defined values.

The update equations read:
\begin{equation}
	\left\lbrace \begin{aligned}
	 \eta_{t+1} &= \eta_t + \alpha \partial_{\theta} \ell_t(\theta_t) \cdot h_t \\
	 h_{t+1} &= \partial_\theta F(\theta_t,\eta_t) \cdot h_t + \partial_\eta F(\theta_t,\eta_t) \cdot \frac{\partial}{\partial e} \eta_t \\
	 \theta_{t+1} &= F \left(\theta_t,\eta_{t+1} \right).
       \end{aligned} \right.
\end{equation}
\end{algo}
\newpage
\section{Computations}
\label{sc} 
\subsection{Computations for Section~\ref{secllrsg}: proof of Fact~\ref{eqrecA}}
\label{cllrsg} 
\begin{proof}
$\theta_0$ is fixed, so $A_0(\eta)=0$.
Let $t \geq 0$. We differentiate \eqref{eqrecthetaf} with respect to $\log(\eta)$, to obtain:
\begin{equation}
\begin{aligned}
  \pa{t+1}{\log \eta} &= \pa{t}{\log \eta} + \frac{\eta}{f(t)} \, \partial_\theta \ell_t(\theta_t) + \frac{\eta}{f(t)} \partial_{\theta}^2 \ell_t(\theta_t(\eta)) \cdot \pa{t}{\log \eta},
\end{aligned}
\end{equation}
which concludes the proof.
\end{proof}

\subsection{Computations for Section~\ref{sec:fllr}} 
\subsubsection{Computations for Section~\ref{secjsg2}: proof of Proposition~\ref{pjsgsg}}
\label{cjsg2} 
To prove Proposition~\ref{pjsgsg}, we use the following three lemmas. The first two are technical, and are used in the proof of the third one, which provides an update formula for the derivative appearing in the statement of the proposition. We may have proceeded without these, as in the proof of Fact~\ref{eqrecA}, but they allow the approach to be more generic. 
\begin{lemma}
\label{lemdF} 
Let
\begin{equation}
\begin{array}{cccc}
	F_t : & \Theta \times \mathbb{R} & \to  \Theta \\
	& (\theta,\eta) & \mapsto & F_t(\theta,\eta)=\theta+\frac{\eta}{f(t)} \partial_\theta \ell_t(\theta).
\end{array}
\end{equation}
Then,
\begin{equation}
  \frac{\partial}{\partial \theta} F_t(\theta,\eta) =  \id + \frac{\eta}{f(t)} \partial_{\theta}^2 \ell_t(\theta)
\end{equation}
and
\begin{equation}
  \frac{\partial}{\partial \eta} F_t(\theta,\eta) = \frac{1}{f(t)} \partial_\theta \ell_t(\theta).
\end{equation}
$\id$ is the identity on the tangent plane to $\Theta$ in $\theta$.
\end{lemma}
\begin{lemma}
\label{lemegFV} 
Let
\begin{equation}
\begin{array}{cccc}
	V_t : & \mathcal{S} & \to   \Theta \times \mathbb{R} \\
       & \et & \mapsto & V_t(\et)=(\theta_t(\eta),\eta_{t+1}). 
\end{array}
\end{equation}
Consider $\log(\et) \in \log(\mathcal{S})$, and any vector $e$ tangent to $\log(\mathcal{S})$ at this point. 
Then the directional derivative of
\begin{equation}
\begin{array}{cccc}
	F_t \circ V_t : & \mathcal{S} & \to   \Theta \\
      & \et & \mapsto & F_t(V_t(\et))=\ti_t(\et)+\frac{\eta_{t+1}}{f(t)} \partial_\theta \ell_t(\ti_t(\et))
\end{array}
\end{equation}
at the point $\log(\et)$ and in the direction $e$ is 
\begin{equation}
  \begin{aligned}
	  \frac{\partial}{\partial e} F_t \circ V_t (\et) &= \pa{t}{e} + \frac{\partial}{\partial e} \eta_{t+1} \, \frac{1}{f(t)} \partial_\theta \ell_t(\ti_t(\et)) + \frac{\eta_{t+1}}{f(t)} \partial_{\theta}^2 \ell_t(\ti_t(\et)) \cdot \pa{t}{e} .
  \end{aligned}
\end{equation}
\end{lemma}
We may then prove the following lemma.
\begin{lemma}
\label{leqjsg2} 
Define
\begin{equation}
	\mathcal{H}_t = \frac{\partial}{\partial e}\ti_t(\et).
\end{equation}
Then for all $t \geq 0$,
\begin{equation}
  \label{eqjsg2} 
  \mh_{t+1}=\mh_t + \frac{\eta_{t+1}}{f(t)} \partial_\theta \ell_t(\ti_t(\et))+ \frac{\eta_{t+1}}{f(t)} \partial_{\theta}^2 \ell_t(\ti_t(\et)) \cdot \mh_t.
\end{equation}
\end{lemma} 
\begin{proof}
The update equation of $\ti_t(\et)$, \eqref{eqrectheta2f}, is such that:
\begin{equation}
  \ti_{t+1}(\et) = \ti_t(\et) + \frac{\eta_{t+1}}{f(t)} \partial_\theta \ell_t(\ti_t(\et)) = F_t \circ V_t(\et). \\
\end{equation}
From the above and Lemma~\ref{lemegFV},
\begin{equation}
\label{eqpd} 
	\pa{t+1}{e}=\frac{\partial}{\partial e} \ti_t(\et) + \frac{\partial}{\partial e} \eta_{t+1} \frac{1}{f(t)} \partial_\theta \ell_t(\ti_t(\et)) + \frac{\eta_{t+1}}{f(t)} \partial_{\theta}^2 \ell_t(\ti_t(\et)) \cdot \frac{\partial}{\partial e} \ti_t(\et).
\end{equation}
Now,
\begin{equation}
\frac{\partial}{\partial e} \eta_{t+1} = \eta_{t+1},
\end{equation}
which concludes the proof.
\end{proof}
Finally, we prove Proposition~\ref{pjsgsg}.
\begin{proof}[Proof of Proposition~\ref{pjsgsg}]
It is sufficient to prove that, for all $t \geq 0$,
\begin{equation}
\frac{\partial}{\partial e} \ell_t(\ti_t(\et)) = \partial_\theta \ell_t(\theta_t) \cdot h_t,
\end{equation}
that is,
\begin{equation}
\partial_\theta \ell_t(\ti_t(\et)) \cdot \mh_t  = \partial_\theta \ell_t(\theta_t) \cdot h_t.
\end{equation}
Therefore, it is sufficient to prove that, for all $t \geq 0$, $\ti_t(\et)=\theta_t$ and $\mh_t=h_t$.
$\ti_0(\eta)=\theta_0$ by construction and, since $\theta_0$ does not depend on $\et$, $\mh_0=0=h_0$.
Assuming the results hold up to iteration $t$, it is straighforward that $\ti_{t+1}(\et)=\theta_{t+1}$, since for all $s \leq t$, $\ti_s(\et)=\theta_s$. Therefore, thanks to Lemma~\ref{leqjsg2}, $\mh_t$ and $h_t$ have the same update, so that $\mh_{t+1}=h_{t+1}$, which concludes the proof.
\end{proof}
\subsubsection{Computations for Section~\ref{secfum}: proof of Proposition~\ref{puhvp}}
\label{cfum} 
\begin{proof}
Thanks to \eqref{eqpd} in Lemma~\ref{lemegFV},
\begin{equation}
\begin{aligned}
	\pa{t+1}{e_{t+1}} &=\frac{\partial}{\partial e_{t+1}} \ti_t(\et) + \frac{\partial}{\partial e_{t+1}} \eta_{t+1} \frac{1}{f(t)} \partial_\theta \ell_t(\ti_t(\et)) \\
			  &+ \frac{\eta_{t+1}}{f(t)} \partial_{\theta}^2 \ell_t(\ti_t(\et)) \cdot \frac{\partial}{\partial e_{t+1}} \ti_t(\et),
\end{aligned}
\end{equation}
that is:
\begin{equation}
	\mathcal{H}_{t+1}=\frac{\partial}{\partial e_{t+1}} \ti_t(\et) + \frac{\partial}{\partial e_{t+1}} \eta_{t+1} \frac{1}{f(t)} \partial_\theta \ell_t(\ti_t(\et)) + \frac{\eta_{t+1}}{f(t)} \partial_{\theta}^2 \ell_t(\ti_t(\et)) \cdot \frac{\partial}{\partial e_{t+1}} \ti_t(\et).
\end{equation}
We first prove:
\begin{equation}
\frac{\partial}{\partial e_{t+1}} \ti_t(\et) = \gamma_{t+1} \pa{t}{e_t}.
\end{equation}
Define $(f_j)_{j \geq 0}$ the canonical basis of the tangent plane to $\log(\mathcal{S})$ at $\et$. Then,
\begin{equation}
e_{t+1} = \gamma_{t+1} (e_t + f_{t+1}).
\end{equation}
Therefore,
\begin{equation}
  \begin{aligned}
\frac{\partial}{\partial e_{t+1}} \ti_t(\et) &= \frac{\partial}{\partial e_{t+1}} \ti_t(\et) \\
&= \gamma_{t+1} \, \pav{t}{e_t} + \pav{t}{f_{t+1}} \\
&= \gamma_{t+1} \pav{t}{e_t}
\end{aligned}
\end{equation}
because the last term is $0$. Therefore, 
\begin{equation}
\frac{\partial}{\partial e_{t+1}} \ti_t(\et) = \gamma_{t+1} \mathcal{H}_t.
\end{equation}
Then, thanks to \eqref{eqpd}, 
\begin{equation}
	\frac{\partial}{\partial e_{t+1}} \eta_{t+1} = \gamma_{t+1} \eta_{t+1},
\end{equation}
which is true since
\begin{equation}
	\frac{\partial}{\partial e_{t+1}} \eta_{t+1} = 
	\gamma_{t+1} \, \frac{\partial}{\partial f_{t+1}} \eta_{t+1} = \gamma_{t+1} \eta_{t+1},
\end{equation}
and concludes the proof.
\end{proof}

\end{document}